\documentclass{amsart}
\usepackage[utf8]{inputenc}
\usepackage{amsfonts}
\usepackage{amsmath}
\usepackage{amssymb}
\usepackage{amsthm}
\usepackage{hyperref}
\usepackage[english]{babel}
\usepackage{url}
\usepackage{verbatim,listings}
\usepackage{tikz,tikz-cd}
\usepackage[margin=1.0in]{geometry}
\newcommand{\pmap}[0]
{
\overline{\odot}_r
}

\newtheorem{thm}{Theorem}[section]
\newtheorem{lem}[thm]{Lemma}

\theoremstyle{definition}
\newtheorem{defn}[thm]{Definition}

\theoremstyle{remark}

\newcommand{\divides}{\,\Big{|}\,}

\title{An Exploration of Crank Generating functions for $t$-core partitions}
\author{Samuel Wilson\\ University of Tennessee - Knoxville} 
\date{}

\begin{document}

\maketitle

\begin{abstract}
    In 1919, Ramanujan discovered his famous congruences for the partition function. Not too long after, Freeman Dyson conjectured a combinatorial statistic existed that explained the three congruences, which he dubbed the \textit{crank}. A crank generating function for the partition function was discovered in 1988 by George Andrews and Frank Garvan. Since then other crank generating functions have been found for many other kinds of partitions. In this paper, we give a family of crank generating functions which explain some partition congruences for $t$-core partitions.
\end{abstract}

\section{Introduction}
The study of Ramanujan-like congruences for partition functions has a deep and rich history. In the early 1900s, Srinivasa Ramanujan proved his three famous congruences for the partition function, $\sum_{n\geq0}p(n)q^n = \prod_{n\geq 1} \frac{1}{1-q^n}$,
\begin{align*}
    p(5n+4) &\equiv 0 \pmod{5} \\
    p(7n+5) &\equiv 0 \pmod{7} \\
    p(11n+6) &\equiv 0 \pmod{11} .
\end{align*}
In the late 1940s, Freeman Dyson defined the \textit{rank} of a partition as the largest part of the partition minus the number of parts. This provided a combinatorial explanation for the congruences mod 5 and 7 listed above by dividing up the partitions of $5n+4$ (resp. $7n+5$) into 5 (resp. 7) equally sized sets, using the partition's rank mod 5 (resp. 7). For a more detailed explanation with examples, see \cite{wilson2024cranks}. Shortly after, Dyson conjectured the existence of a partition statistic that explained all three of Ramanujan's congruences which he dubbed the \textit{crank}. This statistic was found in 1988 by George Andrews and Frank Garvan, and its definition is given below.

\begin{defn}\label{defn:rank}
 Let $\lambda$ be an integer partition, $l(\lambda)$ be the largest part, $\omega(\lambda)$ be the number of $1$'s in $\lambda$, and $\mu(\lambda)$ be the number of parts in $\lambda$ that are larger than $\omega(\lambda)$. Then
 $$
 crank(\lambda):=
 \begin{cases}
     l(\lambda) & \text{ if } \omega(\lambda) = 0 \\
     \mu(\lambda) - \omega(\lambda) & \text{ if } \omega(\lambda) > 0,
 \end{cases}
 $$

with the generating function given by
$$
C(z,\tau) := \sum_{m=-\infty}^\infty\sum_{n=0}^\infty M(m,n)\zeta^mq^n = \prod_{n=1}^\infty \frac{1-q^n}{(1-\zeta q^n)(1-\zeta^{-1}q^n)},
$$
 where $q=e^{2\pi i \tau}$, $\zeta = e^{2\pi i z}$, and $M(m,n)$ is the number of partitions of $n$ with crank $m$.
 \end{defn}
 It should be noted that the validity of the crank can be seen from its generating function. Setting $\zeta = 1$ recovers the generating function for $p(n)$. In addition, we have that $$\Phi_5(\zeta) \divides [q^{5n+4}]C(z,\tau)$$ for all $n\geq 0$, where $[q^n]C(z,\tau)$ denotes the coefficient of $q^n$ in $C(z,\tau)$ and $\Phi_5(\zeta)$ is the 5th-cyclotomic polynomial. This criterion ensures that the partitions of $n$ are distributed equally among the residue classes mod $5$. Similarly, we have
$$\Phi_7(\zeta) \divides [q^{7n+5}]C(z,\tau)$$
and 
$$\Phi_{11}(\zeta) \divides [q^{11n+6}]C(z,\tau).$$
More generally, given a generating function $\sum_{n\geq 0 }a(n)q^n$  satisfying $a(\alpha n+\beta) \equiv 0 \pmod{\ell}$ for some positive integers $\alpha,\beta$, some prime $\ell$, and a power series $C(z,\tau) \in \mathbb{Z}[[q,\zeta]]$, we say the power series \textit{explains the congruence} if and only if 
$$
C(0,\tau) = \sum_{n\geq 0 }a(n)q^n
$$
and 
$$
\Phi_\ell(\zeta) \divides [q^{\alpha n+\beta}]C(z,\tau).
$$
When this is the case for at least 1 congruence, we say $C(z,\tau)$ is a \textit{crank generating function}. In 2020, Rolen, Tripp, and Wagner discovered a family of crank generating functions for $p_k(n)$, the partitions of $n$ into $k$-colors \cite{rolen2021cranks}. Their result was given by 
$$
C_k(z_1,\ldots z_{\lfloor\frac{k+1}{2}\rfloor};\tau) := C(0;\tau)^{\lfloor\frac{k}{2}\rfloor}\prod_{i=1}^{\lfloor\frac{k+1}{2}\rfloor}C(z_i,\tau),
$$
where $z_i = a_iz$ in which the $a_i \in \mathbb{Z}$ are chosen with respect to the congruence being proven. \\
More recently, the author discovered 3 families of crank generating functions for 
$$
\sum_{n\geq0}p_{k,j}(n)q^n = \prod_{n\geq1}\frac{(1+q^n)^j}{(1-q^n)^k},
$$ the partitions of $n$ into $(k+j)$-colors, where $j$-colors have distinct parts \cite{wilson2024cranks}. In this case, the three families came from the pairs $(k,j) = (2,\ell m+(\ell-1)), (3,\ell m+(\ell-3)),$ and $(\ell,\ell m+(\ell-3))$, where $\ell$ is a prime and $m$ is a positive integer. \\

The aim of this paper is to extend the notion of crank generating functions to $t$-core partitions, the generating function of which is given by
$$
\sum_{n\geq 0}p^{(t)}(n)q^n = \prod_{n\geq 1} \frac{(1-q^{tn})^t}{(1-q^n)}.
$$
Recently, Meher and Jindal proved some congruences for $p^{(t)}(n)$ using the theory of modular forms in \cite{meher2023arithmeticdensitycongruencestcore}. First, we give the following crank generating function, which explains the congruences for $p^{(5)}(n).$ For convenience, we set $(1-\zeta^{\pm k}q^n) =(1-\zeta^{k}q^n)(1-\zeta^{-k}q^n)$.

\begin{thm}\label{thm:5crank}
Let $p^{(5)}(n)$ denote the $5$-core partition function. Then,
    $$
    C^{(5)}(z,\tau) :=  \prod_{n=1}^\infty\frac{(1-q^n)(1-q^{5n})(1-\zeta^{\pm2}q^{5n})(1-\zeta^{\pm 4}q^{5n})}{(1-\zeta^{\pm1}q^n)}
    $$
    is a crank generating function for $p^{(5)}(n)$ which explains the congruences $p^{(5)}(15n+\beta) \equiv 0 \pmod{3}$, where $\beta \in \{6,10,12,13\}$.
\end{thm}

We also give the following crank generating functions which explain some of the congruences proven by Meher and Jindal for $p^{(t)}(n)$ for $t >5$.

\begin{thm}\label{thm:multcrank}
    Let $p^{(t)}(n)$ denote the $t$-core partition function. Then,
    $$
    C^{(7)}(z,\tau) =  \prod_{n=1}^\infty\frac{(1-q^n)(1-q^{7n})^3(1-\zeta^{\pm2}q^{7n})(1-\zeta^{\pm 4}q^{7n})}{(1-\zeta^{\pm1}q^n)}, \\
    $$
    explains $p^{(7)}(21n+\beta) \equiv 0 \pmod{3}$, where $\beta \in \{8,11,17\}$,
    $$
    C^{(11)}(z,\tau) =  \prod_{n=1}^\infty\frac{(1-q^n)(1-q^{11n})^3(1-\zeta^{\pm2}q^{11n})(1-\zeta^{\pm 4}q^{11n})(1-\zeta^{\pm 8}q^{11n})(1-\zeta^{\pm 10}q^{11n})}{(1-\zeta^{\pm1}q^n)},
    $$
    explains $p^{(11)}(33n+\beta) \equiv 0 \pmod{3}$, where $\beta \in \{11, 20, 26, 29, 32\}$,
    $$
    C^{(17)}(z,\tau) =  \prod_{n=1}^\infty\frac{(1-q^n)(1-q^{17n})^5(1-\zeta^{\pm2}q^{17n})(1-\zeta^{\pm 4}q^{17n})(1-\zeta^{\pm 8}q^{17n})(1-\zeta^{\pm 10}q^{17n})(1-\zeta^{\pm 14}q^{17n})(1-\zeta^{\pm 16}q^{17n})}{(1-\zeta^{\pm1}q^n)},
    $$
    explains $p^{(17)}(51n+\beta) \equiv 0 \pmod{3}$, where $\beta \in \{14,20,23,26,35,38,41,47\}$, and
    $$
    C^{(19)}(z,\tau) =  \prod_{n=1}^\infty\frac{(1-q^n)(1-q^{19n})^7(1-\zeta^{\pm2}q^{19n})(1-\zeta^{\pm 4}q^{19n})(1-\zeta^{\pm 8}q^{19n})(1-\zeta^{\pm 10}q^{19n})(1-\zeta^{\pm 14}q^{19n})(1-\zeta^{\pm 16}q^{19n})}{(1-\zeta^{\pm1}q^n)}
    $$
    explains $p^{(19)}(57n+\beta) \equiv 0 \pmod{3}$, where $\beta \in \{14,17,26,35,38,41,44,50,56\}$.
\end{thm}
\textit{Remark.} We note a couple of facts. First, each $C^{(t)}(z,\tau)$, for $t > 5$, only explain half of the congruences that appear in the literature for $p^{(t)}(n)$. Second, there are congruences for $13$-core and $23$-core partitions for which we do not have a crank generating function. A description for these is given in the concluding section. Also, while numerical evidence suggests there may be congruences for $t>23$, these have not yet appeared in the literature. Hence we concern ourselves with $t \leq 23$ for now. \\

We now give an outline for the rest of the paper. In Section 2, we briefly describe Radu's lemma and the necessary prerequisites. Section 3 contains the proofs for Theorems \ref{thm:5crank} and \ref{thm:multcrank}. We will then conclude with a discussion on the crank generating functions and some questions.

\section{Radu's Lemma}

In order to prove our main theorems, we will utilize a lemma of Radu, Lemma \ref{mainlemma}. To do so, we will first define some prerequisites needed for the statement of the lemma. \\

Following Radu and Sellers in \cite{RaduSellers}, let $R(M)$ be the set of integer sequences indexed by $\{\delta_1,...,\delta_k\}$, the divisors of $M$, and let $r \in R(M)$ be defined by $r:=(r_{\delta_1},r_{\delta_2},...,r_{\delta_k})$. Also, let $\alpha$ be a positive integer. Define $[s]_\alpha$ to be the equivalence class of $s \pmod{\alpha}$,
and let $\mathbb{S}_\alpha \subset \mathbb{Z}^*_\alpha$ be the set of all squares in $\mathbb{Z}^*_\alpha$, the invertible elements of $\mathbb{Z}_\alpha$. We now give a few necessary definitions.
\begin{defn}\label{defn:P(t)}
    Let $\alpha$ and $M$ be positive integers, $r \in R(M)$, and $\beta\in\{0,1,\ldots,\alpha-1\}.$ We define the map $\overline{\odot}_r: \mathbb{S}_{24\alpha} \times \{0,1,\ldots,\alpha-1\} \to \{0,1,\ldots,\alpha-1\}$ by 
    $$
    [s]_{24\alpha} \overline{\odot}_r \beta:=\beta s+\frac{s-1}{24}\sum_{\delta | M}\delta r_\delta \pmod{\alpha}.
    $$
    Also, define 
    $$
    P_{\alpha,r}(\beta) := \{[s]_{24\alpha} \pmap \beta \,\,|\,\, [s]_{24\alpha} \in \mathbb{S}_{24\alpha}\}.
    $$
\end{defn}

\begin{defn}\label{defn:gamma}
    Let $\Gamma$ be defined as the modular group,
    $$
    \Gamma := SL_2(\mathbb{Z}) = \left\{ \begin{pmatrix} a &b \\c &d\end{pmatrix}\, , \, a,b,c,d \in \mathbb{Z} \text{ and } ad-bc = 1\right\}.
    $$
    Similarly, let $\Gamma_0(N)$ and $\Gamma_\infty$ be defined as
    $$
    \Gamma_0(N) := \left\{ \begin{pmatrix} a &b \\c &d\end{pmatrix} \in \Gamma \,,\, N|c \right\}
    $$
    and
    $$
    \Gamma_{\infty} := \left\{ \begin{pmatrix} 1 & b \\0 &1\end{pmatrix} \,,\, b \in \mathbb{Z}\right\}.
    $$
\end{defn}
It is well known that
$$
[\Gamma:\Gamma_0(N)] = N\prod_{p |N}\left(1+\frac{1}{p}\right).
$$

We will also need the following lemma from \cite{WANG_2017}.

\begin{lem}
    Suppose $N$ or $\frac{N}{2}$ is square-free. Then
    $$
    \bigcup_{\delta|N}\Gamma_0(N) \begin{pmatrix}
        1&0 \\\delta &1 \end{pmatrix} \Gamma_{\infty} = \Gamma.
    $$
\end{lem}

We now state the conditions needed for Lemma \ref{mainlemma}.

Let $M\in \mathbb{Z}^+$ and $r_\delta \in \mathbb{Z}$ for all $\delta \divides M$. Suppose one has a $q$-series given by 
$$
\prod_{\delta|M}\prod_{n\geq 1}(1-q^{\delta n})^{r_{\delta}} = \sum_{n\geq 0} a(n)q^n,
$$
which we wish to show satisfies the family of congruences $a(\alpha n+\beta') = 0$ for all $\beta' \in P_{\alpha,r}(\beta)$ and for all non-negative integers $n$.  \\

Define $\Delta^*$ by the following. Let $(\alpha,M,N,r,\beta)$ be in $\Delta^*$ if and only if 
\begin{enumerate}
    \item $\alpha,M,$ and $N$ are positive integers, $r \in R(M)$, and $\beta\in \{0,1,\ldots,\alpha-1\}$.
    \item For all primes $p$, $p|\alpha$ implies $p|N$.
    \item For all $\delta |M$ such that $r_{\delta} \neq 0$, $\delta|\alpha N$.
    \item $\kappa N \displaystyle\sum_{\delta|M}r_\delta \frac{\alpha N}{\delta} \equiv 0 \pmod{24}$.
    \item $\kappa N \displaystyle\sum_{\delta|M}r_\delta \equiv 0 \pmod{8}$.
    \item $\displaystyle{\frac{24}{\gcd(\kappa(-24\beta-\sum_{\delta|M}\delta r_{\delta}),24\alpha)}}\divides N$.
    \item $(s,j) = \pi(M,(r_\delta))$ implies $(4|\kappa N$ and $8|sN)$ or $(2|s$ and $8|N(1-j))$ if $2|\alpha$.
\end{enumerate}
where $\kappa = gcd(\alpha^2-1,24)$, and $\pi(M,(r_\delta)) = (s,j)$ is defined by $\prod_{\delta|M}\delta^{|r_\delta|} = 2^sj$, where $s$ is a positive integer and $j$ is odd. Also, for $\gamma := \begin{pmatrix}
    a & b \\ c& d
\end{pmatrix}$ define

$$
p_{a}(\gamma):= \min_{\lambda \in \{1,2,\ldots,\alpha-1\}} \frac{1}{24}\sum_{\delta |M} r_\delta \frac{\gcd^2(\delta(a+\kappa \lambda c)\alpha c)}{\delta \alpha}
$$

and

$$
p_a^*(\gamma) := \frac{1}{24}\sum_{\delta|M}\frac{a_\delta \gcd^2(\delta,c)}{\delta}.
$$

A special case of Radu's lemma is now given by the following.

\begin{lem}\label{mainlemma}
Let $(\alpha, M, N, r,\beta) \in \Delta^*, a=(a_\delta) \in R(N)$ and $\gamma := \{\gamma_1,\ldots,\gamma_n\} \subset \Gamma$ be a complete list of double-coset representatives in $\Gamma_0(N)\backslash\Gamma/\Gamma_{\infty}$. Suppose $p_{a}(\gamma_i)+p_a^*(\gamma_i)\geq 0$ for all $\gamma_i \in \gamma$, and $\beta_{min} := \min\{\beta' \,\,|\,\, \beta' \in P_{\alpha,r}(\beta)\}$. If
$$
\nu :=\frac{1}{24}\left(\left(\sum_{\delta|N}a_\delta \,+\,\sum_{\delta|M}r_{\delta}\right) [\Gamma:\Gamma_0(N)]-\sum_{\delta|N}\delta a_\delta\right)- \frac{1}{24\alpha}\sum_{\delta |M}\delta r_\delta - \frac{\beta_{min}}{\alpha},
$$
then
$$
\sum_{n=0}^{\lfloor \nu \rfloor}a(\alpha n+\beta')q^n = 0 
$$
for all $\beta' \in P_{\alpha,r}(\beta)$ implies
$$
\sum_{n\geq0}a(\alpha n+\beta')q^n = 0
$$
for all $\beta' \in P_{\alpha,r}(\beta)$.
\end{lem}

The benefit of Lemma \ref{mainlemma} is twofold. First, it applies to any quotient of the form 
$$
\prod_{\delta|M}\prod_{n\geq 1}(1-q^{\delta n})^{r_{\delta}} = \sum_{n\geq 0} a(n)q^n,
$$
not just modular forms. Secondly, in many cases it greatly reduces the upper bound compared to Sturm's theorem. \\

\textit{Remark:} Note that Lemma \ref{mainlemma} as stated is only useful in proving where $\sum_{n\geq0}a(\alpha n+\beta')q^n = 0$. In fact, Radu's lemma is more general, allowing us to show when $\sum_{n\geq0}a(\alpha n+\beta')q^n \equiv 0 \pmod{u}$, for any positive integer $u$. As we are only interested in the case when our generating functions are identically 0, we give only the version of the theorem that is necessary.

\section{Main Theorems}
In order to prove the main theorems, it suffices to show that $[q^{3t n+\beta}]C^{(t)}(z,\tau) \equiv 0 \pmod{\Phi_3(\zeta)}$ for all pairs $(t,\beta)$ covered by Theorems \ref{thm:5crank} and \ref{thm:multcrank}. \\

It is readily verifiable that we have the following congruences.
$$
C^{(t)}(z,\tau) \equiv \prod_{k\geq 1}\frac{(1-q^k)^2}{(1-q^{3k})(1-q^{tk})} \cdot (1-q^{3tk})^j\pmod{\Phi_3(\zeta)} \
$$
if $t \equiv 2 \pmod{3}$, and 
$$
C^{(t)}(z,\tau) \equiv \prod_{n \geq 1}\frac{(1-q^k)^2(1-q^{tk})}{(1-q^{3k})} \cdot (1-q^{3tk})^j\pmod{\Phi_3(\zeta)}
$$
if $t \equiv 1 \pmod{3}$, where $j$ is a positive integer that depends on $t$.  \\

Note that since we are concerned with the coefficients $[q^{3t n+\beta}]C^{(t)}(z,\tau)$, we may ignore the $(1-q^{3tk})^j$ term in the above products. In particular, it suffices to use Lemma \ref{mainlemma} to show 
$$
[q^{3tn+\beta}]\prod_{k\geq 1}\frac{(1-q^k)^2}{(1-q^{3k})(1-q^{tk})} = 0 \quad \text{ or } \quad [q^{3tn+\beta}]\prod_{k\geq 1}\frac{(1-q^k)^2(1-q^{tk})}{(1-q^{3k})} = 0,
$$ 
depending on the value of $t \pmod{3}$. \\

All of the following calculations have been verified using MAPLE. \\

\textit{Remark} It should be noted that the proof of Theorem \ref{thm:5crank} differs slightly from the proof of Theorem \ref{thm:multcrank}, because it is the only example where neither $N$ nor $N/2$ is square-free. Therefore we must manually find the double-coset representatives. Otherwise, the proofs are analogous.

\begin{proof}[Proof of Theorem \ref{thm:5crank}]
   Note that $(15,15,45,(2,-1,1,0),6)$ lies in $\Delta^*$. Thus we may apply Lemma \ref{mainlemma}. To do so, we need to find $a \in R(45)$, and a set of double-coset representatives for $\Gamma_0(45)\backslash\Gamma/\Gamma_{\infty}$ such that $P_{\alpha,r}(\gamma_i)+p_a^*(\gamma_i)\geq 0$. Recall that the equivalence classes of cusps in $\Gamma_0(N)$ are in bijection with the desired double-coset representatives (see \cite{MR2112196} for example). The corresponding cusps for $\Gamma_0(45)$ are 
   $$
   \left\{ \infty,0,\frac{1}{3},\frac{1}{5},\frac{1}{9},\frac{1}{15},\frac{2}{15},\frac{4}{45} \right\},
   $$
   These give the following double-coset representatives,
   $$
   \left\{  
   \begin{pmatrix} 1 &0 \\0&1 \end{pmatrix}, 
   \begin{pmatrix} 0 &-1 \\1&0 \end{pmatrix},
   \begin{pmatrix} 1 &0 \\3&1 \end{pmatrix}, 
   \begin{pmatrix} 1 &0 \\5&1 \end{pmatrix}, 
   \begin{pmatrix} 1 &0 \\9&1 \end{pmatrix}, 
   \begin{pmatrix} 1 &0 \\15&1 \end{pmatrix}, 
   \begin{pmatrix} 2 &-1 \\15&-7 \end{pmatrix}, 
   \begin{pmatrix} 4 &-1 \\45&-11 \end{pmatrix}
   \right\}.
   $$
   Choosing $a = (2, 0, 0, 0, 0, 0)$ satisfies the desired conditions and yields $\lfloor \nu \rfloor = 11$. The last condition is verified easily in MAPLE.
\end{proof}

\begin{proof}[Proof of Theorem \ref{thm:multcrank}]
    The following tables summarize the tuples needed to satisfy the conditions of Lemma \ref{mainlemma}, thus proving the desired congruences.
    \begin{center}
    \begin{table}[h!]
    \begin{tabular}{| c c c c c c c c|} 
     \hline
     $t$-core & $\alpha$ & $M$ & $N$ & $r$ & $P_{\alpha,r}(\beta)$& $a$ & $\lfloor \nu \rfloor$ \\ [0.5ex] 
     \hline\hline
     7 & 21 & 21 & 42 & (2,-1,1,0) & $P_{\alpha,r}(8)$ = \{8,11,17\}& (2,0,0,0,0,0,0,0) & 15 \\ 
     11 & 33 & 33 & 33 & (2,-1,-1,0) & $P_{\alpha,r}(11) = \{11,20,26,29,32\}$& (4,0,0,0) & 7 \\
    17 & 51 & 51 & 51 & (2,-1,-1,0) & $P_{\alpha,r}(14) = \{14,20,23,26,35,38,41,47\}$& (6,0,0,0) & 17 \\
    19 & 57 & 57 & 57 & (2,-1,1,0) & $P_{\alpha,r}(14) = \{14,17,26,35,38,41,44,50,56\}$& (6,0,0,0) & 26 \\
    \hline
    \end{tabular}
    \caption{Tuple values for the various crank generating functions.}
    \label{table:1}
    \end{table}
    \end{center}

\end{proof}

\section{Conclusion and Discussion}
Regarding the remark in Section 1, we note that the $t$-core partition functions considered (for $t > 5$) actually satisfy more congruences than our crank generating functions explain (see \cite{meher2023arithmeticdensitycongruencestcore}).  Note that in their proofs, they also use Radu's lemma, but find that $7$-core partitions additionally satisfy congruences for $P_{\alpha,r}(3) = \{3,15,18\}$. This trend continues for all $t$-cores, where $t \in \{7,11,13,17,19,23\}$. It turns out, our crank generating functions only account for congruences given by $P_{\alpha,r}(\beta)$, where $\beta$ is a quadratic non-residue, mod $3$. \\

We conclude with two questions that are of great interest.
\begin{enumerate}
    \item Can a family of crank generating functions be given that explain all congruences for $p^{(t)}(n)$ simultaneously when $t \leq 23$?
    \item Are there other congruences that belong in this ``mod 3" family for $t>23$?
    \item Does the family mentioned in (1) or the crank generating functions defined in the paper extend to an infinite family?
\end{enumerate}

\section{Statement of Funding, Conflicts of Interest, and Data Availability}
This work received no funding, has no conflicts of interest, and there is no necessary data to be made available.

\bibliographystyle{alpha}
\bibliography{refs}

@misc{rolen2021cranks,
      title={Cranks for Ramanujan-type congruences of $k$-colored partitions}, 
      author={Larry Rolen and Zack Tripp and Ian Wagner},
      year={2021},
      eprint={2006.16195},
      archivePrefix={arXiv},
      primaryClass={math.NT}
}

@article {meher2023arithmeticdensitycongruencestcore,
    AUTHOR = {Jindal, Ankita and Meher, N. K.},
     TITLE = {Arithmetic density and congruences of {$t$}-core partitions},
   JOURNAL = {Results Math.},
  FJOURNAL = {Results in Mathematics},
    VOLUME = {79},
      YEAR = {2024},
    NUMBER = {1},
     PAGES = {Paper No. 4, 23},
      ISSN = {1422-6383,1420-9012},
   MRCLASS = {11P83 (05A17 11F11)},
  MRNUMBER = {4658656},
MRREVIEWER = {James\ A.\ Sellers},
       DOI = {10.1007/s00025-023-02032-z},
       URL = {https://doi.org/10.1007/s00025-023-02032-z},
}

@article {wilson2024cranks,
    AUTHOR = {Wilson, Samuel},
     TITLE = {A proposed crank for {$(k+j)$}-colored partitions, with {$j$}
              colors having distinct parts},
   JOURNAL = {Res. Number Theory},
  FJOURNAL = {Research in Number Theory},
    VOLUME = {11},
      YEAR = {2025},
    NUMBER = {2},
     PAGES = {Paper No. 46, 8},
      ISSN = {2522-0160,2363-9555},
   MRCLASS = {11P83 (05A15)},
  MRNUMBER = {4886091},
MRREVIEWER = {Jeremy\ Lovejoy},
       DOI = {10.1007/s40993-025-00625-x},
       URL = {https://doi.org/10.1007/s40993-025-00625-x},
}

@article{RaduSellers,
author = {RADU, SILVIU and SELLERS, JAMES A.},
title = {CONGRUENCE PROPERTIES MODULO 5 AND 7 FOR THE pod FUNCTION},
journal = {International Journal of Number Theory},
volume = {07},
number = {08},
pages = {2249-2259},
year = {2011},
doi = {10.1142/S1793042111005064},

URL = { 
    
        https://doi.org/10.1142/S1793042111005064
    
    

}}

@article{WANG_2017, title={ARITHMETIC PROPERTIES OF $(k,\ell )$ -REGULAR BIPARTITIONS}, volume={95}, DOI={10.1017/S0004972716000964}, number={3}, journal={Bulletin of the Australian Mathematical Society}, author={WANG, LIUQUAN}, year={2017}, pages={353–364}}

@book {MR2112196,
    AUTHOR = {Diamond, Fred and Shurman, Jerry},
     TITLE = {A first course in modular forms},
    SERIES = {Graduate Texts in Mathematics},
    VOLUME = {228},
 PUBLISHER = {Springer-Verlag, New York},
      YEAR = {2005},
     PAGES = {xvi+436},
      ISBN = {0-387-23229-X},
   MRCLASS = {11Fxx},
  MRNUMBER = {2112196},
MRREVIEWER = {Henri\ Darmon},
}

\end{document}